\documentclass{amsart}
\usepackage{amsmath,amssymb,amsthm}
\newtheorem{theorem}{Theorem}[section]
\newtheorem{lemma}[theorem]{Lemma}
\newtheorem{proposition}[theorem]{Proposition}

\theoremstyle{definition}

\theoremstyle{remark}
\newtheorem{remark}[theorem]{Remark}
\newtheorem*{Acknowledgments}{Acknowledgments}
\numberwithin{equation}{section}
\begin{document}
%%%%%%%%%%%%%%%%%%%%%%%%%%%%%%%%%%%%%%%%%%%%%%%%%%%%%%%%%%%%%%%%%%%%%%%%%%%%

\title[Factorizations of EP operators]{Factorizations of EP operators}
%%%%%%%%%%%%%%%%%%%%%%%%%%%%%%%%%%%%%%%%%%%%%%%%%%%%%%%%%%%%%%%%%%%%%%%%%%%%

\author{Dimosthenis Drivaliaris}
\address{Department of Financial and Management Engineering\\
University of the Aegean
31, Fostini Str.\\
82100 Chios\\
Greece} \email{d.drivaliaris@fme.aegean.gr}
%%%%%%%%%%%%%%%%%%%%%%%%%%%%%%%%%%%%%%%%%%%%%%%%%%%%%%%%%%%%%%%%%%%%%%%%%%%%

\author{Sotirios Karanasios}
\address{Department of Mathematics\\
National Technical University of Athens\\
Zografou Campus\\
15780 Zografou\\
Greece} \email{skaran@math.ntua.gr}
%%%%%%%%%%%%%%%%%%%%%%%%%%%%%%%%%%%%%%%%%%%%%%%%%%%%%%%%%%%%%%%%%%%%%%%%%%%%

\author{Dimitrios Pappas}
\address{Department of Statistics and Insurance Science\\
University of Piraeus\\
80 Karaoli \& Dimitriou Street\\
18534 Piraeus\\
Greece} \email{dimpap@unipi.gr}
%%%%%%%%%%%%%%%%%%%%%%%%%%%%%%%%%%%%%%%%%%%%%%%%%%%%%%%%%%%%%%%%%%%%%%%%%%%%

\date{}
%%%%%%%%%%%%%%%%%%%%%%%%%%%%%%%%%%%%%%%%%%%%%%%%%%%%%%%%%%%%%%%%%%%%%%%%%%%%

\begin{abstract}
In this paper we characterize EP operators through the existence of different types of factorizations. Our results extend to EP operators existing characterizations for EP matrices and give new characterizations both for EP matrices and EP operators.
\end{abstract}
%%%%%%%%%%%%%%%%%%%%%%%%%%%%%%%%%%%%%%%%%%%%%%%

%%%%%%%%%%%%%%%%%%%%%%%%%%%%%%%%%%%%%%%%%%%%%%%%%%%%%%%%%%%%%%%%%%%%%%%%%%%%

%%%%%%%%%%%%%%%%%%%%%%%%%%%%%%%%%%%%%%%%%%%%%%%%%%%%%%%%%%%%%%%%%%%%%%%%%%%%

\maketitle
%%%%%%%%%%%%%%%%%%%%%%%%%%%%%%%%%%%%%%%%%%%%%%%%%%%%%%%%%%%%%%%%%%%%%%%%%%%%

\section{Introduction}
%%%%%%%%%%%%%%%%%%%%%%%%%%%%%%%%%%%%%%%%%%%%%%%%%%%%%%%%%%%%%%%%%%%%%%%%%%%%

A square matrix $T$ is called EP (or range Hermitian) if $\mathcal{N}(T)=\mathcal{N}(T^*)$. EP matrices were introduced by Schwerdtfeger in \cite{Schwerdtfeger}. Ever since many authors have studied EP matrices with entries from $\mathbb{C}$ or from an arbitrary field (for more on EP matrices see \cite[Chapter 4.4]{Israel} and \cite[Chapter 4]{Campbell1}). The notion of EP matrices was extended by Campbell and Meyer to operators with closed range on a Hilbert space in \cite{Campbell}. For further results on EP operators see \cite{Brock, Djordjevic1, Djordjevic2, Djordjevic3, Fildan, Lesnjak, Spitkovsky}.
%%%%%%%%%%%%%%%%%%%%%%%%%%%%%%%%%%%%%%%%%%%%%%%%%%%%%%%%%%%%%%%%%%%%%%%%%%%%

One of the main reasons for studying EP matrices is that, as it was proved by Pearl in \cite{Pearl3}, a matrix $T$ is EP if and only if it commutes with its Moore-Penrose inverse $T^{\dag}$. Using that as a definition the notion of EP matrices was extended to elements of $C^*$-algebras in \cite{Harte, Koliha1, Koliha3}, to elements of rings with involution in \cite{Hartwig2, Hartwig3}, to elements of semigroups with involution in \cite{Hartwig1} and to elements of unital Banach algebras and operators on a Banach space in \cite{Boasso}.
%%%%%%%%%%%%%%%%%%%%%%%%%%%%%%%%%%%%%%%%%%%%%%%%%%%%%%%%%%%%%%%%%%%%%%%%%%%%

One of the main trends in the study of EP matrices and EP operators has been their characterization through factorizations. Pearl showed in \cite{Pearl1, Pearl2} that a matrix $T$ is EP if and only if it can be written in the form $U(A\oplus 0)U^*$, with $U$ unitary or isomorphism and $A$ an isomorphism. Although the fact that an EP operator can be written in the form $U(A\oplus 0)U^*$ has been used in the study of EP operators (see for example \cite{Campbell} and \cite{Djordjevic1}), characterizations of EP operators through factorizations of that form have not been studied. In Section 3 we characterize EP operators through factorizations of the form $U(A\oplus 0)U^*$ and through simultaneous factorizations of $T$ and of $T^*$ or of $T$ and of $T^{\dag}$ or of $T^*T$ and of $TT^*$ of similar forms. In \cite{Pearl1} Pearl showed that a matrix $T$ is EP if and only if there exists an isomorphism $V$ such that $T^*=VT$. This result was extended to operators on a Hilbert space by Fildan in \cite{Fildan} (see also \cite{Brock}), to elements of $C^*$-algebras by Harte and Mbekhta in \cite{Harte} and by Koliha in \cite{Koliha3} and to elements of rings by Hartwig in \cite{Hartwig2, Hartwig3}. In Section 4 we show that actually $V$ can be replaced with an injective operator and give other characterizations of  EP operators of that type. In \cite{Pearl3} Pearl proved that if $T$ is a matrix and $B$ and $C$ are matrices with $B$ injective, $C$ surjective and $T=BC$, then $T$ is EP if and only if there exists an isomorphism $K$ such that $C=KB^*$. In Section 5 we characterize EP operators through factorizations of that kind. Note that many of the results in Sections 3, 4 and 5 also give new characterizations for EP matrices.
%%%%%%%%%%%%%%%%%%%%%%%%%%%%%%%%%%%%%%%%%%%%%%%%%%%%%%%%%%%%%%%%%%%%%%%%%%%%

\section{Preliminaries and notation}
%%%%%%%%%%%%%%%%%%%%%%%%%%%%%%%%%%%%%%%%%%%%%%%%%%%%%%%%%%%%%%%%%%%%%%%%%%%%

Throughout $\mathcal{H}$ and $\mathcal{K}$ will denote complex Hilbert spaces, $\mathcal{B}(\mathcal{H},\mathcal{K})$ will denote the algebra of all bounded linear operators from $\mathcal{H}$ to $\mathcal{K}$ and $\mathcal{B}(\mathcal{H})$ will denote $\mathcal{B}(\mathcal{H},\mathcal{H})$. The words operator and subspace will mean bounded linear operator and closed linear subspace respectively. An  $n\times n$ matrix will always be identified with an operator on $\mathbb{C}^n$ with its standard basis. If $T\in\mathcal{B}(\mathcal{H})$, then we will denote its kernel by $\mathcal{N}(T)$, its range by $\mathcal{R}(T)$ and its adjoint by $T^*$. We will say that $U\in\mathcal{B}(\mathcal{H},\mathcal{K})$ is an isomorphism if it is injective and onto and that $U\in\mathcal{B}(\mathcal{H},\mathcal{K})$ is unitary if $U^*U=I$ and $UU^*=I$.
%%%%%%%%%%%%%%%%%%%%%%%%%%%%%%%%%%%%%%%%%%%%%%%%%%%%%%%%%%%%%%%%%%%%%%%%%%%%

We will denote the standard orthonormal basis of $l^2(\mathbb{N})$ by $\{ e_0, e_1, e_2, \ldots\}$. Whenever we refer to the left and the right shift we will mean the left and the right shift on $l^2(\mathbb{N})$.
%%%%%%%%%%%%%%%%%%%%%%%%%%%%%%%%%%%%%%%%%%%%%%%%%%%%%%%%%%%%%%%%%%%%%%%%%%%%

To denote the direct sum of two subspaces $\mathcal{M}_1$ and $\mathcal{M}_2$ of $\mathcal{H}$ we will use $\mathcal{M}_1\oplus\mathcal{M}_2$ and to denote the orthogonal sum of two orthogonal subspaces $\mathcal{M}_1$ and $\mathcal{M}_2$ of $\mathcal{H}$ we will use $\mathcal{M}_1\displaystyle{\mathop{\oplus}^{\perp}}\mathcal{M}_2$. The orthogonal projection onto a subspace $\mathcal{M}$ of $\mathcal{H}$ will be denoted by $P_\mathcal{M}$. If $T\in\mathcal{B}(\mathcal{H})$ with closed range, then we will denote $P_{\mathcal{R}(T)}$ by $P_T$. By $\mathcal{H}\oplus\mathcal{K}$ we will denote the exterior 2-sum of $\mathcal{H}$ and $\mathcal{K}$. If $T\in\mathcal{B}(\mathcal{H})$ and $S\in\mathcal{B}(\mathcal{K})$, then by $T\oplus S$ we will denote their direct sum acting on $\mathcal{H}\oplus\mathcal{K}$.
%%%%%%%%%%%%%%%%%%%%%%%%%%%%%%%%%%%%%%%%%%%%%%%%%%%%%%%%%%%%%%%%%%%%%%%%%%%%

The Moore-Penrose inverse of an operator $T\in\mathcal{B}(\mathcal{H},\mathcal{K})$ with closed range is the unique operator $T^{\dag}\in\mathcal{B}(\mathcal{K},\mathcal{H})$ satisfying the following four conditions:
$$
TT^{\dag}T=T,\,\, T^{\dag}TT^{\dag}=T^{\dag},\,\, (TT^{\dag})^*=TT^{\dag},\,\, (T^{\dag}T)^*=T^{\dag}T.
$$
Standard references on generalized inverses are the books of Ben-Israel and Greville \cite{Israel}, Campbell and Meyer \cite{Campbell1} and Groetsch \cite{Groetsch}. It is immediate from the definition of the Moore-Penrose inverse that $\mathcal{R}(T^{\dag})=\mathcal{R}(T^*)$, $\mathcal{N}(T^{\dag})=\mathcal{N}(T^*)$, $P_T=TT^{\dag}$ and $P_{T^*}=T^{\dag}T$.
%%%%%%%%%%%%%%%%%%%%%%%%%%%%%%%%%%%%%%%%%%%%%%%%%%%%%%%%%%%%%%%%%%%%%%%%%%%%

An operator $T$ with closed range is called EP if $\mathcal{N}(T)=\mathcal{N}(T^*)$. It is easy to see that
\begin{equation}
\label{def}
T \text{ EP}\Leftrightarrow\mathcal{R}(T)=\mathcal{R}(T^*)\Leftrightarrow\displaystyle{\mathcal{R}(T)\mathop{\oplus}^{\perp}\mathcal{N}(T)=\mathcal{H}}\Leftrightarrow TT^{\dag}=T^{\dag}T.
\end{equation}
%%%%%%%%%%%%%%%%%%%%%%%%%%%%%%%%%%%%%%%%%%%%%%%%%%%%%%%%%%%%%%%%%%%%%%%%%%%%

In the finite dimensional case if $\mathcal{N}(T)$ is contained in $\mathcal{N}(T^*)$ or vice versa, then $T$ is EP (this is not in general true in the infinite dimensional case; for example, let $T$ be the right or the left shift). Using this observation one can immediately see that for matrices many of the results in Sections 3, 4 and 5 hold with weaker assumptions.
%%%%%%%%%%%%%%%%%%%%%%%%%%%%%%%%%%%%%%%%%%%%%%%%%%%%%%%%%%%%%%%%%%%%%%%%%%%%

Note that the second equivalence in (\ref{def}) is not true if the sum is not an orthogonal one. Actually we have that $\mathcal{R}(T)\oplus\mathcal{N}(T)=\mathcal{H}$ if and only if the Drazin index of $T$ is equal to 0 or 1 (see \cite[Exercise 4.3.5]{Israel}).
%%%%%%%%%%%%%%%%%%%%%%%%%%%%%%%%%%%%%%%%%%%%%%%%%%%%%%%%%%%%%%%%%%%%%%%%%%%%

Obviously $T$ is EP if and only if $T^*$ is EP and $T$ is EP if and only if $T^{\dag}$ is EP. Moreover if $T$ is EP, then $P_{T^*}T=T$ and $TP_{T}=T$.
%%%%%%%%%%%%%%%%%%%%%%%%%%%%%%%%%%%%%%%%%%%%%%%%%%%%%%%%%%%%%%%%%%%%%%%%%%%%

Isomorphisms are EP. Moreover we have that if $T$ is EP, then $T$ is an isomorphism if and only if it is injective or surjective. Normal operators with closed range are EP. A bounded projection is EP if and only if it is orthogonal.
%%%%%%%%%%%%%%%%%%%%%%%%%%%%%%%%%%%%%%%%%%%%%%%%%%%%%%%%%%%%%%%%%%%%%%%%%%%%

A well-known result about EP operators (see \cite{Djordjevic1}) is the following:
%%%%%%%%%%%%%%%%%%%%%%%%%%%%%%%%%%%%%%%%%%%%%%%%%%%%%%%%%%%%%%%%%%%%%%%%%%%%

\begin{lemma}
\label{lem1} Let $T\in\mathcal{B}(\mathcal{H})$ with closed range.
If $T$ is EP, then
$$T|_{\mathcal{R}(T)}:\mathcal{R}(T)\rightarrow\mathcal{R}(T)$$
is an isomorphism.
\end{lemma}
%%%%%%%%%%%%%%%%%%%%%%%%%%%%%%%%%%%%%%%%%%%%%%%%%%%%%%%%%%%%%%%%%%%%%%%%%%%%

\section{Factorizations of the form $T=U(A\oplus 0)U^*$}
%%%%%%%%%%%%%%%%%%%%%%%%%%%%%%%%%%%%%%%%%%%%%%%%%%%%%%%%%%%%%%%%%%%%

In this section we will characterize EP operators through factorizations of $T$ of the form $T=U(A\oplus 0)U^*$ and through simultaneous factorizations of $T$ and $T^*$ or of $T$ and $T^{\dag}$ or of $T^*T$ and $TT^*$ of similar forms.
%%%%%%%%%%%%%%%%%%%%%%%%%%%%%%%%%%%%%%%%%%%%%%%%%%%%%%%%%%%%%%%%%%%%

We start with two elementary lemmas which we will use in the proofs of these characterizations. The proof of the first one is elementary.
%%%%%%%%%%%%%%%%%%%%%%%%%%%%%%%%%%%%%%%%%%%%%%%%%%%%%%%%%%%%%%%%%%%%

\begin{lemma}
\label{lem2a}
Let $T_1\in\mathcal{B}(\mathcal{H}_1)$ and $T_2\in\mathcal{B}(\mathcal{H}_2)$ with closed range. Then $T_1\oplus T_2$ is EP if and only if $T_1$ and $T_2$ are EP.
\end{lemma}
%%%%%%%%%%%%%%%%%%%%%%%%%%%%%%%%%%%%%%%%%%%%%%%%%%%%%%%%%%%%%%%%%%%%

\begin{lemma}
\label{lem2}
Let $T\in\mathcal{B}(\mathcal{H})$ and $G\in\mathcal{B}(\mathcal{K})$ with closed range and $U\in\mathcal{B}(\mathcal{K},\mathcal{H})$ injective such that $T=UGU^*$. Then $T$ is EP if and only if $G$ is EP.
\end{lemma}
%%%%%%%%%%%%%%%%%%%%%%%%%%%%%%%%%%%%%%%%%%%%%%%%%%%%%%%%%%%%%%%%%%%%

\begin{proof}
If $T$ is EP, then, by the injectivity of $U$, we get that $G(\mathcal{R}(U^*))=G^*(\mathcal{R}(U^*))$. Since $\mathcal{R}(U^*)$ is dense in $\mathcal{K}$ and $G$ has closed range, that implies $\mathcal{R}(G)=\mathcal{R}(G^*)$. On the other hand if $G$ is EP, then $\mathcal{N}(G)=\mathcal{N}(G^*)$. Combining that with the injectivity of $U$ we get that $\mathcal{N}(T)=\mathcal{N}(T^*)$.
\end{proof}
%%%%%%%%%%%%%%%%%%%%%%%%%%%%%%%%%%%%%%%%%%%%%%%%%%%%%%%%%%%%%%%%%%%%

\begin{remark}
\label{lem2b} (a) The result of the previous lemma is not true if $U$ is not injective. To see that, for the one direction let
$$T=\left[
\begin{array}{cc}
1&0\\
0&0
\end{array}
\right],\,\,
G=\left[
\begin{array}{ccc}
0&1&0\\
0&1&0\\
0&0&0
\end{array}
\right],\,\,
U=\left[
\begin{array}{ccc}
0&1&0\\
0&0&1
\end{array}
\right]
$$
and for the other one let
$$T=\left[
\begin{array}{cc}
0&1\\
0&0
\end{array}
\right],\,\,
G=\left[
\begin{array}{ccc}
0&1&0\\
0&0&1\\
1&0&0
\end{array}
\right],\,\,
U=\left[
\begin{array}{ccc}
0&1&0\\
0&0&1
\end{array}
\right].
$$
(b) An immediate corollary of the previous lemma is that if $G\in\mathcal{B}(\mathcal{K})$ is EP and $U\in\mathcal{B}(\mathcal{K},\mathcal{H})$ is unitary, then $UGU^*=UGU^{-1}$ is EP. On the other hand if $G\in\mathcal{B}(\mathcal{K})$ is EP and
$U\in\mathcal{B}(\mathcal{K},\mathcal{H})$ is an isomorphism, then it is not in general true that $UGU^{-1}$ is EP. To see that let $\mathcal{M}_1$ and $\mathcal{M}_2$ be subspaces of $\mathcal{H}$, with $\mathcal{M}_1^{\perp}\ne\mathcal{M}_2$ and $\mathcal{M}_1\oplus\mathcal{M}_2=\mathcal{H}$, $\mathcal{K}$ be the exterior 2-sum of $\mathcal{M}_1$ and $\mathcal{M}_2$, $G:\mathcal{K}\rightarrow\mathcal{K}$, with $G=I_{\mathcal{M}_1}\oplus 0$, and $U:\mathcal{K}\rightarrow\mathcal{H}$, with $U(y,z)=y+z$, for all $(y,z)\in \mathcal{K}$.
\end{remark}
%%%%%%%%%%%%%%%%%%%%%%%%%%%%%%%%%%%%%%%%%%%%%%%%%%%%%%%%%%%%%%%%%%%%

In the following theorem we characterize EP operators through factorizations of the form $T=U(A\oplus 0)U^*$.
%%%%%%%%%%%%%%%%%%%%%%%%%%%%%%%%%%%%%%%%%%%%%%%%%%%%%%%%%%%%%%%%%%%%

\begin{theorem}
\label{prop1} Let $T\in\mathcal{B}(\mathcal{H})$ with closed range. Then the following are equivalent:
\begin{enumerate}
\item $T$ is EP.
\item There exist Hilbert spaces $\mathcal{K}_1$ and $\mathcal{L}_1$, $U_1\in\mathcal{B}(\mathcal{K}_1\oplus \mathcal{L}_1,\mathcal{H})$ unitary and $A_1\in\mathcal{B}(\mathcal{K}_1)$ isomorphism such that $T=U_1(A_1\oplus 0)U_1^*$.
\item There exist Hilbert spaces $\mathcal{K}_2$ and $\mathcal{L}_2$, $U_2\in\mathcal{B}(\mathcal{K}_2\oplus \mathcal{L}_2,\mathcal{H})$ isomorphism and $A_2\in\mathcal{B}(\mathcal{K}_2)$ isomorphism such that $T=U_2(A_2\oplus 0)U_2^*$.
\item There exist Hilbert spaces $\mathcal{K}_3$ and $\mathcal{L}_3$, $U_3\in\mathcal{B}(\mathcal{K}_3\oplus \mathcal{L}_3,\mathcal{H})$ injective and
$A_3\in\mathcal{B}(\mathcal{K}_3)$ isomorphism such that $T=U_3(A_3\oplus 0)U_3^*$.
\end{enumerate}
\end{theorem}
%%%%%%%%%%%%%%%%%%%%%%%%%%%%%%%%%%%%%%%%%%%%%%%%%%%%%%%%%%%%%%%%%%%%

\begin{proof}
(1)$\Rightarrow$(2): Let $\mathcal{K}_1=\mathcal{R}(T)$, $\mathcal{L}_1=\mathcal{N}(T)$, $U_1:\mathcal{K}_1\oplus\mathcal{L}_1\rightarrow\mathcal{H}$, with
$$U_1(y,z)=y+z,$$
for all $y\in\mathcal{R}(T)$ and $z\in\mathcal{N}(T)$, and $A_1=T|_{\mathcal{R}(T)}:\mathcal{R}(T)\rightarrow\mathcal{R}(T)$. Since $T$ is EP, $\displaystyle{\mathcal{R}(T)\mathop{\oplus}^{\perp}\mathcal{N}(T)=\mathcal{H}}$ and thus $U_1$ is unitary. Moreover it is easy to see that
$$U_1^*x=(P_Tx,P_{\mathcal{N}(T)}x),$$
for all $x\in \mathcal{H}$.  From Lemma \ref{lem1}, $A_1$ is an isomorphism. Using $TP_T=T$, which follows from $T$ being EP, we get that
$$T=U_1(A_1\oplus 0)U_1^*.$$
(2)$\Rightarrow$(3)$\Rightarrow$(4) is obvious and (4)$\Rightarrow$(1) follows from Lemmas \ref{lem2a} and \ref{lem2}.
\end{proof}
%%%%%%%%%%%%%%%%%%%%%%%%%%%%%%%%%%%%%%%%%%%%%%%%%%%%%%%%%%%%%%%%%%%%

\begin{remark}
\label{prop1a} (a) What we said in Remark \ref{lem2b}(a) shows that if in the previous theorem we don't assume that $U_3$ is injective, then $T$ is not in general EP.\\
(b) It is easy to see that if $T=U(A\oplus 0)U^*\in\mathcal{B}(\mathcal{H})$, with $U\in\mathcal{B}(\mathcal{K}\oplus \mathcal{L},\mathcal{H})$ unitary and $A\in\mathcal{B}(\mathcal{K})$ an isomorphism, then $T^{\dag}=U(A^{-1}\oplus 0)U^*$.
\end{remark}
%%%%%%%%%%%%%%%%%%%%%%%%%%%%%%%%%%%%%%%%%%%%%%%%%%%%%%%%%%%%%%%%%%%%

If in (2) of Theorem \ref{prop1} we just assume that $A_1$ is injective with closed range, then $T$ is not in general EP. In the following proposition we show that the existence of simultaneous factorizations of $T$ and $T^*$ of the form $T=U(A\oplus 0)U^*$ and $T^*=U(B\oplus 0)U^*$, with $U$, $A$ and $B$ injective, implies that $T$ is EP.
%%%%%%%%%%%%%%%%%%%%%%%%%%%%%%%%%%%%%%%%%%%%%%%%%%%%%%%%%%%%%%%%%%%%

\begin{proposition}
\label{prop1b} Let $T\in\mathcal{B}(\mathcal{H})$ with closed range. Then the following are equivalent:
\begin{enumerate}
\item $T$ is EP.
\item
\begin{enumerate}
\item There exist Hilbert spaces $\mathcal{K}_1$ and
$\mathcal{L}_1$, $V_1\in\mathcal{B}(\mathcal{K}_1\oplus
\mathcal{L}_1,\mathcal{H})$ injective,
$W_1\in\mathcal{B}(\mathcal{K}_1\oplus
\mathcal{L}_1,\mathcal{H})$,
$S_1\in\mathcal{B}(\mathcal{H},\mathcal{K}_1\oplus
\mathcal{L}_1)$, $A_1\in\mathcal{B}(\mathcal{K}_1)$ injective and
$B_1\in\mathcal{B}(\mathcal{K}_1)$ such that $T=V_1(A_1\oplus
0)S_1$ and $T^*=W_1(B_1\oplus 0)S_1$. \item There exist Hilbert
spaces $\mathcal{K}_2$ and $\mathcal{L}_2$,
$V_2\in\mathcal{B}(\mathcal{K}_2\oplus
\mathcal{L}_2,\mathcal{H})$,
$W_2\in\mathcal{B}(\mathcal{K}_2\oplus \mathcal{L}_2,\mathcal{H})$
injective, $S_2\in\mathcal{B}(\mathcal{H},\mathcal{K}_2\oplus
\mathcal{L}_2)$, $A_2\in\mathcal{B}(\mathcal{K}_2)$, and
$B_2\in\mathcal{B}(\mathcal{K}_2)$ injective such that
$T=V_2(A_2\oplus 0)S_2$ and $T^*=W_2(B_2\oplus 0)S_2$
\end{enumerate}
\end{enumerate}
\end{proposition}
%%%%%%%%%%%%%%%%%%%%%%%%%%%%%%%%%%%%%%%%%%%%%%%%%%%%%%%%%%%%%%%%%%%%

\begin{proof}
(1)$\Rightarrow$(2) follows from Theorem \ref{prop1}.\\
(2)$\Rightarrow$(1): First assume that (a) holds. From $T=V_1(A_1\oplus 0)S_1$ and the injectivity of $V_1$ and $A_1$ we get
$$
\mathcal{N}(T)=S_1^{-1}\left(\{ 0\}\oplus\mathcal{L}_1\right).
$$
By $T^*=W_1(B_1\oplus 0)S_1$ we get that
$$
S_1^{-1}\left(\{ 0\}\oplus\mathcal{L}_1\right)\subseteq\mathcal{N}(T^*).
$$
So
$$
\mathcal{N}(T)\subseteq\mathcal{N}(T^*).
$$
Similarly, by (b), we get $\mathcal{N}(T^*)\subseteq\mathcal{N}(T)$. Therefore $T$ is EP.
\end{proof}
%%%%%%%%%%%%%%%%%%%%%%%%%%%%%%%%%%%%%%%%%%%%%%%%%%%%%%%%%%%%%%%%%%%%

\begin{remark}
\label{prop1c}
(a) If in (2)(a) we just assume that $V_1$ is injective or that $A_1$ is injective, then $T$ is not in general EP.\\
(b) Using Remark \ref{prop1a}(b) and $\mathcal{N}(T^{\dag})=\mathcal{N}(T^*)$ we get that the results of the previous proposition also hold if we replace $T^*$ with $T^{\dag}$.
\end{remark}
%%%%%%%%%%%%%%%%%%%%%%%%%%%%%%%%%%%%%%%%%%%%%%%%%%%%%%%%%%%%%%%%%%%%

What we said in Remark \ref{lem2b}(b) shows that if in (3) of Theorem \ref{prop1} we replace $T=U_2(A_2\oplus 0)U_2^*$ with
$T=U_2(A_2\oplus 0)U_2^{-1}$, then $T$ is not in general EP (actually if $T=U(A\oplus 0)U^{-1}$, with $U$ and $A$ isomorphisms, then $T$ has Drazin index equal to 0 or 1; see \cite[p. 1728]{Djordjevic1}). In the following proposition we show that the existence of simultaneous factorizations of $T$ and $T^*$ of the form $T=U(A\oplus 0)U^{-1}$ and $T^*=U(B\oplus 0)U^{-1}$, with $U$ an isomorphism and one of $A$ and $B$ an isomorphism, implies that $T$ is EP.
%%%%%%%%%%%%%%%%%%%%%%%%%%%%%%%%%%%%%%%%%%%%%%%%%%%%%%%%%%%%%%%%%%%%

\begin{proposition}
\label{prop1d}
Let $T\in\mathcal{B}(\mathcal{H})$ with closed range. Then the following are equivalent:
\begin{enumerate}
\item $T$ is EP.
\item There exist Hilbert spaces $\mathcal{K}_1$ and $\mathcal{L}_1$, $U_1\in\mathcal{B}(\mathcal{K}_1\oplus \mathcal{L}_1,\mathcal{H})$ isomorphism, $A_1\in\mathcal{B}(\mathcal{K}_1)$ isomorphism and $B_1\in\mathcal{B}(\mathcal{K}_1)$ such that $T=U_1(A_1\oplus 0)U^{-1}_1$ and $T^*=U_1(B_1\oplus
0)U_1^{-1}$.
\end{enumerate}
\end{proposition}
%%%%%%%%%%%%%%%%%%%%%%%%%%%%%%%%%%%%%%%%%%%%%%%%%%%%%%%%%%%%%%%%%%%%

\begin{proof}
(1)$\Rightarrow$(2) follows from Theorem \ref{prop1}.\\
(2)$\Rightarrow$(1): As in the proof of (2)$\Rightarrow$(1) in Proposition \ref{prop1b} we get that
$$
\mathcal{N}(T)\subseteq\mathcal{N}(T^*).
$$
Taking adjoints in (2) we get
$$
T^*=(U_1^*)^{-1}(A_1^*\oplus 0)U_1^* \text{ and } T=(U_1^*)^{-1}(B_1^*\oplus 0)U_1^*,
$$
which gives us in the same manner
$
\mathcal{N}(T^*)\subseteq\mathcal{N}(T)
$
and so $T$ is EP.
\end{proof}
%%%%%%%%%%%%%%%%%%%%%%%%%%%%%%%%%%%%%%%%%%%%%%%%%%%%%%%%%%%%%%%%%%%%

\begin{remark}
\label{prop1e}
Using Remark \ref{prop1a}(b), $\mathcal{N}(T^{\dag})=\mathcal{N}(T^*)$ and $\mathcal{N}((T^*)^{\dag})=\mathcal{N}(T)$ we get that the results of the previous proposition also hold if we replace $T^*$ with $T^{\dag}$.
\end{remark}
%%%%%%%%%%%%%%%%%%%%%%%%%%%%%%%%%%%%%%%%%%%%%%%%%%%%%%%%%%%%%%%%%%%%

We continue with a characterization of EP operators through simultaneous factorizations of $T^*T$ and $TT^*$.
%%%%%%%%%%%%%%%%%%%%%%%%%%%%%%%%%%%%%%%%%%%%%%%%%%%%%%%%%%%%%%%%%%%%

\begin{proposition}
\label{prop2}
Let $T\in\mathcal{B}(\mathcal{H})$ with closed range. Then the following are equivalent:
\begin{enumerate}
\item $T$ is EP.
\item There exist Hilbert spaces $\mathcal{K}_1$ and $\mathcal{L}_1$, $U_1\in\mathcal{B}(\mathcal{K}_1\oplus \mathcal{L}_1,\mathcal{H})$ unitary and $A_1\in\mathcal{B}(\mathcal{K}_1)$ isomorphism such that $T^*T=U_1(A_1^*A_1\oplus 0)U_1^*$ and $TT^*=U_1(A_1A_1^*\oplus 0)U_1^*$.
\item
\begin{enumerate}
\item There exist Hilbert spaces $\mathcal{K}_2$ and
$\mathcal{L}_2$, $V_2\in\mathcal{B}(\mathcal{K}_2\oplus
\mathcal{L}_2,\mathcal{H})$ injective,
$W_2\in\mathcal{B}(\mathcal{K}_2\oplus
\mathcal{L}_2,\mathcal{H})$,
$S_2\in\mathcal{B}(\mathcal{H},\mathcal{K}_2\oplus
\mathcal{L}_2)$, $A_2\in\mathcal{B}(\mathcal{K}_2)$ injective and
$B_2\in\mathcal{B}(\mathcal{K}_2)$ such that $T^*T=V_2(A_2\oplus
0)S_2$ and $TT^*=W_2(B_2\oplus 0)S_2$. \item There exist Hilbert
spaces $\mathcal{K}_3$ and $\mathcal{L}_3$,
$V_3\in\mathcal{B}(\mathcal{K}_3\oplus
\mathcal{L}_3,\mathcal{H})$,
$W_3\in\mathcal{B}(\mathcal{K}_3\oplus \mathcal{L}_3,\mathcal{H})$
injective, $S_3\in\mathcal{B}(\mathcal{H},\mathcal{K}_3\oplus
\mathcal{L}_3)$, $A_3\in\mathcal{B}(\mathcal{K}_3)$, and
$B_3\in\mathcal{B}(\mathcal{K}_3)$ injective such that
$T^*T=V_3(A_3\oplus 0)S_3$ and $TT^*=W_3(B_3\oplus 0)S_3$.
\end{enumerate}
\item There exist Hilbert spaces $\mathcal{K}_4$ and $\mathcal{L}_4$, $U_4\in\mathcal{B}(\mathcal{K}_4\oplus \mathcal{L}_4,\mathcal{H})$ isomorphism and $A_4, B_4\in\mathcal{B}(\mathcal{K}_4)$ injective with closed range such that $T^*T=U_4(A_4\oplus 0)U_4^{-1}$ and $TT^*=U_4(B_4\oplus 0)U_4^{-1}$.
\end{enumerate}
\end{proposition}
%%%%%%%%%%%%%%%%%%%%%%%%%%%%%%%%%%%%%%%%%%%%%%%%%%%%%%%%%%%%%%%%%%%%

\begin{proof}
(1)$\Rightarrow$(2) follows from Theorem \ref{prop1}. (2)$\Rightarrow$(3), (2)$\Rightarrow$(4) and (4)$\Rightarrow$(3) are obvious. (3)$\Rightarrow$(1) can be proved in a manner similar to the proof of (2)$\Rightarrow$(1) in Proposition \ref{prop1b} using $\mathcal{N}(T^*T)=\mathcal{N}(T)$ and $\mathcal{N}(TT^*)=\mathcal{N}(T^*)$.
\end{proof}
%%%%%%%%%%%%%%%%%%%%%%%%%%%%%%%%%%%%%%%%%%%%%%%%%%%%%%%%%%%%%%%%%%%%

\begin{remark}
\label{prop2a}
(a) If we only assume that one of the conditions in (2) holds, then $T$ is not in general EP. To see that let $T=\left[
\begin{smallmatrix} 0&1\\0&0\end{smallmatrix}\right]$.\\
(b) If there exist Hilbert spaces $\mathcal{K}$ and $\mathcal{L}$, $U\in\mathcal{B}(\mathcal{K}\oplus\mathcal{L},\mathcal{H})$ unitary, $A\in\mathcal{B}(\mathcal{K})$ isomorphism and $B\in\mathcal{B}(\mathcal{K})$ such that $T^*T=U(A\oplus 0)U^*$ and $TT^*=U(B\oplus 0)U^*$,
then $T$ is not in general EP. To see that let $T$ be the right shift, $\mathcal{K}=l^2(\mathbb{N})$, $\mathcal{L}=\{ 0\}$, $U=I_{\mathcal{K}}\oplus I_{\mathcal{L}}$, $A=I_{\mathcal{K}}$ and $B=P_{\overline{span}\left\{e_1, e_2, e_3, \ldots\right\}}$.
\end{remark}
%%%%%%%%%%%%%%%%%%%%%%%%%%%%%%%%%%%%%%%%%%%%%%%%%%%%%%%%%%%%%%%%%%%%

\section{Factorizations of the form $T^*=ST$}
%%%%%%%%%%%%%%%%%%%%%%%%%%%%%%%%%%%%%%%%%%%%%%%%%%%%%%%%%%%%%%%%%%%%

In this section we will characterize EP operators through factorizations of the form $T^*=ST$, of the form $T^{\dag}=ST$, of the form $T^*T=STT^*$, of the form $T^*T=TST^*$, of the form $T^{\dag}T=STT^{\dag}$ and of the form $T^{\dag}T=TST^{\dag}$.
%%%%%%%%%%%%%%%%%%%%%%%%%%%%%%%%%%%%%%%%%%%%%%%%%%%%%%%%%%%%%%%%%%%%

We start with characterizations of EP operators via factorizations of the form $T^*=ST$.
%%%%%%%%%%%%%%%%%%%%%%%%%%%%%%%%%%%%%%%%%%%%%%%%%%%%%%%%%%%%%%%%%%%%

\begin{proposition}
\label{prop3} Let $T\in\mathcal{B}(\mathcal{H})$ with closed range. Then the following are equivalent:
\begin{enumerate}
\item $T$ is EP.
\item There exists an isomorphism $V\in\mathcal{B}(\mathcal{H})$ such that $T^*=VT$.
\item There exists $N\in\mathcal{B}(\mathcal{H})$ injective such that $T^*=NT$.
\item There exist $S_1, S_2\in\mathcal{B}(\mathcal{H})$ such that $T^*=S_1T$ and $T=S_2T^*$.
\end{enumerate}
\end{proposition}
%%%%%%%%%%%%%%%%%%%%%%%%%%%%%%%%%%%%%%%%%%%%%%%%%%%%%%%%%%%%%%%%%%%%

\begin{proof}
(1)$\Rightarrow$(2): By Theorem \ref{prop1}, $T=U(A\oplus 0)U^*$, with $U\in\mathcal{B}(\mathcal{K}\oplus\mathcal{L},\mathcal{H})$ unitary and $A\in\mathcal{B}(\mathcal{K})$ an isomorphism. If we take
$$V=U(A^*A^{-1}\oplus I_{\mathcal{L}})U^*:\mathcal{H}\rightarrow\mathcal{H},$$
then $V$ is an isomorphism with $ T^*=VT$.\\
(2)$\Rightarrow$(3) is obvious and (2)$\Rightarrow$(4) follows from $T^*=VT\Rightarrow T=V^{-1}T^*$.\\
(3)$\Rightarrow$(1): By $T^*=NT$ we get that $\mathcal{N}(T)\subseteq\mathcal{N}(T^*)$. On the other hand, by $T^*=NT$ and the injectivity of $N$ we get that $\mathcal{N}(T^*)\subseteq\mathcal{N}(T)$ and so $T$ is EP.\\
(4)$\Rightarrow$(1): By $T^*=S_1T$ we get that $\mathcal{N}(T)\subseteq\mathcal{N}(T^*)$
and by $T=S_2T^*$ we get that $\mathcal{N}(T^*)\subseteq\mathcal{N}(T)$ and so $T$ is EP.
\end{proof}
%%%%%%%%%%%%%%%%%%%%%%%%%%%%%%%%%%%%%%%%%%%%%%%%%%%%%%%%%%%%%%%%%%%%%%%%%%%%

\begin{remark}
\label{prop3a}
(a) The equivalence of (1) and (2) was proved in \cite[Th\'{e}or\'{e}me 7]{Fildan}. The equivalence of (1), (2) and (4) also follows from the more general results concerning $C^*$-algebras in \cite[Theorem 10]{Harte} and in \cite[Theorem 3.1]{Koliha3}. We give our proof since it differs from the ones in those papers. Note that \cite[Theorem 10]{Harte} and \cite[Theorem 3.1]{Koliha3} contain other conditions equivalent to (2).\\
(b) Obviously the result also holds if in (2) in the place of $T^*=VT$ we put $T=VT^*$ or $T=T^*V$ or $T^*=TV$. Moreover the result also holds if in (3) in the place of $T^*=NT$ we put $T=NT^*$. On the other hand we cannot in the place of $T^*=NT$ put $T^*=TN$. To see that let $T$ be the left shift and $N=(T^*)^2$.\\
(c) If $T=U(A\oplus 0)U^*\in\mathcal{B}(\mathcal{H})$, with $U\in\mathcal{B}(\mathcal{K}\oplus\mathcal{L},\mathcal{H})$ unitary and $A\in\mathcal{B}(\mathcal{K})$ an isomorphism, then the solutions of the equation $T^*=XT$ are
$$X=U
\left[
\begin{array}{cc}
A^*A^{-1}&B\\
0&D
\end{array}
\right]U^*,\;\;B\in\mathcal{B}(\mathcal{L},\mathcal{K}),
D\in\mathcal{B}(\mathcal{L}).$$
For similar results about matrices
see \cite{Katz3}.
\end{remark}
%%%%%%%%%%%%%%%%%%%%%%%%%%%%%%%%%%%%%%%%%%%%%%%%%%%%%%%%%%%%%%%%%%%%%%%%%%%%

We continue with characterizations of EP operators via factorizations of the form $T^{\dag}=ST$.
%%%%%%%%%%%%%%%%%%%%%%%%%%%%%%%%%%%%%%%%%%%%%%%%%%%%%%%%%%%%%%%%%%%%

\begin{proposition}
\label{prop3b} Let $T\in\mathcal{B}(\mathcal{H})$ with closed range. Then the following are equivalent:
\begin{enumerate}
\item $T$ is EP.
\item There exists an isomorphism $V\in\mathcal{B}(\mathcal{H})$ such that $T^{\dag}=VT=TV$.
\item There exists $N\in\mathcal{B}(\mathcal{H})$ injective such that $T^{\dag}=NT$.
\item There exist $S_1, S_2\in\mathcal{B}(\mathcal{H})$ such that $T^{\dag}=S_1T$ and $T=S_2T^{\dag}$.
\end{enumerate}
\end{proposition}
%%%%%%%%%%%%%%%%%%%%%%%%%%%%%%%%%%%%%%%%%%%%%%%%%%%%%%%%%%%%%%%%%%%%

\begin{proof}
(1)$\Rightarrow$(2): By Theorem \ref{prop1}, $T=U(A\oplus 0)U^*$, with $U\in\mathcal{B}(\mathcal{K}\oplus\mathcal{L},\mathcal{H})$ unitary and $A\in\mathcal{B}(\mathcal{K})$ an isomorphism. Using Remark \ref{prop1a}(b) we get that if
$$
V=U(A^{-2}\oplus
I_{\mathcal{L}})U^*:\mathcal{H}\rightarrow\mathcal{H},$$
then $V$ is an isomorphism with $T^{\dag}=VT=TV$.\\
The rest follows in a manner similar to the proof of Proposition \ref{prop3}, since $\mathcal{N}(T^{\dag})=\mathcal{N}(T^*)$.
\end{proof}
%%%%%%%%%%%%%%%%%%%%%%%%%%%%%%%%%%%%%%%%%%%%%%%%%%%%%%%%%%%%%%%%%%%%%%%%%%%%

\begin{remark}
\label{prop3c}
(a) The equivalence of (1) and (4) also follows from the more general results concerning $C^*$-algebras in \cite[Theorem 3.1]{Koliha3}. Moreover the equivalence of (1), (2) and (4) also follows from the results concerning unital Banach algebras in \cite[Theorem 18]{Boasso} (see also \cite[Theorem 16]{Boasso}).\\
(b) In general it is not true that if $T$ is an EP operator and $V$ is an isomorphism such that $T^{\dag}=VT$, then $T^{\dag}=TV$. To see that let
$$
T=\left[
\begin{array}{cc}
1&0\\
0&0
\end{array}
\right] \text{ and }
V=\left[
\begin{array}{cc}
1&1\\
0&1
\end{array}
\right].
$$
(c) If $T=U(A\oplus 0)U^*\in\mathcal{B}(\mathcal{H})$, with $U\in\mathcal{B}(\mathcal{K}\oplus\mathcal{L},\mathcal{H})$ unitary and $A\in\mathcal{B}(\mathcal{K})$ an isomorphism, then the solutions of the equation $T^{\dag}=XT$ are
$$X=U
\left[
\begin{array}{cc}
A^{-2}&B\\
0&D
\end{array}
\right]U^*,\;\;B\in\mathcal{B}(\mathcal{L},\mathcal{K}),
D\in\mathcal{B}(\mathcal{L}),$$
and the solutions of the equation $T^{\dag}=XT=TX$ are
$$X=U
\left[
\begin{array}{cc}
A^{-2}&0\\
0&D
\end{array}
\right]U^*,\;\;D\in\mathcal{B}(\mathcal{L}).$$
\end{remark}
%%%%%%%%%%%%%%%%%%%%%%%%%%%%%%%%%%%%%%%%%%%%%%%%%%%%%%%%%%%%%%%%%%%%%%%%%%%%

We finish this section with two characterizations of EP operators through factorizations of the form $T^*T=STT^*$ and $T^*T=TST^*$.
%%%%%%%%%%%%%%%%%%%%%%%%%%%%%%%%%%%%%%%%%%%%%%%%%%%%%%%%%%%%%%%%%%%%%%%%%%%%

\begin{proposition}
\label{prop4} Let $T\in\mathcal{B}(\mathcal{H})$ with closed range. Then the following are equivalent:
\begin{enumerate}
\item $T$ is EP.
\item There exists an isomorphism $V\in\mathcal{B}(\mathcal{H})$ such that $T^*T=VTT^*$.
\item There exists $N\in\mathcal{B}(\mathcal{H})$ injective such that $T^*T=NTT^*$.
\item There exist $S_1, S_2\in\mathcal{B}(\mathcal{H})$ such that $T^*T=S_1TT^*$ and $TT^*=S_2T^*T$.
\end{enumerate}
\end{proposition}
%%%%%%%%%%%%%%%%%%%%%%%%%%%%%%%%%%%%%%%%%%%%%%%%%%%%%%%%%%%%%%%%%%%%%%%%%%%%

\begin{proof}
(1)$\Rightarrow$(2): By Theorem \ref{prop1}, $T=U(A\oplus 0)U^*$, with $U\in\mathcal{B}(\mathcal{K}\oplus\mathcal{L},\mathcal{H})$ unitary and $A\in\mathcal{B}(\mathcal{K})$ an isomorphism. If we take
$$V=U(A^*A(A^*)^{-1}A^{-1}\oplus I_{\mathcal{L}})U^*:\mathcal{H}\rightarrow\mathcal{H},$$
then $V$ is an isomorphism with $ T^*T=VTT^*$.\\
The rest follows in a manner similar to the proof of Proposition \ref{prop3}, since $\mathcal{N}(T^*T)=\mathcal{N}(T)$ and $\mathcal{N}(TT^*)=\mathcal{N}(T^*)$.
\end{proof}
%%%%%%%%%%%%%%%%%%%%%%%%%%%%%%%%%%%%%%%%%%%%%%%%%%%%%%%%%%%%%%%%%%%%%%%%%%%%

\begin{remark}
\label{prop4a}
(a) If in (3) we put $T^*T=TT^*N$ in the place of $T^*T=NTT^*$, then $T$ is not in general EP. To see that let $L$ be the left shift and $R$ be the right shift, $T=L\oplus 0$ and
$$
N=\left[
\begin{array}{cc}
P_{\overline{\mathrm{span}}\{ e_1, e_2, \ldots\}}&0\\
P_{{\mathrm{span}}\{ e_0\}}&R
\end{array}
\right].
$$
(b) If $T=U(A\oplus 0)U^*\in\mathcal{B}(\mathcal{H})$, with $U\in\mathcal{B}(\mathcal{K}\oplus \mathcal{L},\mathcal{H})$ unitary and $A\in\mathcal{B}(\mathcal{K})$ an isomorphism, then the solutions of the equation $T^*T=XTT^*$ are
$$X=U
\left[
\begin{array}{cc}
A^*A(A^*)^{-1}A^{-1}&B\\
0&D
\end{array}
\right]U^*,\;\;B\in\mathcal{B}(\mathcal{L},\mathcal{K}),
D\in\mathcal{B}(\mathcal{L}).$$
\end{remark}
%%%%%%%%%%%%%%%%%%%%%%%%%%%%%%%%%%%%%%%%%%%%%%%%%%%%%%%%%%%%%%%%%%%%%%%%%%%%

Using Theorem \ref{prop1}, Remark \ref{prop1a}(b), $\mathcal{N}(T^{\dag}T)=\mathcal{N}(T)$ and $\mathcal{N}(TT^{\dag})=\mathcal{N}(T^*)$ we can show that the results of the previous proposition also hold if we replace $T^*T$ with $T^{\dag}T$ and $TT^*$ with $TT^{\dag}$ and if we just replace $T^*T$ with $T^{\dag}T$ or $TT^*$ with $TT^{\dag}$.
%%%%%%%%%%%%%%%%%%%%%%%%%%%%%%%%%%%%%%%%%%%%%%%%%%%%%%%%%%%%%%%%%%%%%%%%%%%%

\begin{proposition}
\label{prop5} Let $T\in\mathcal{B}(\mathcal{H})$ with closed range. Then the following are equivalent:
\begin{enumerate}
\item $T$ is EP.
\item There exists an isomorphism $V\in\mathcal{B}(\mathcal{H})$ such that $T^*T=TV^*VT^*$.
\item There exists $N\in\mathcal{B}(\mathcal{H})$ injective such that $T^*T=TN^*NT^*$.
\item There exist $S_1, S_2\in\mathcal{B}(\mathcal{H})$ such that $T^*T=TS_1T^*$ and $TT^*=T^*S_2T$.
\end{enumerate}
\end{proposition}
%%%%%%%%%%%%%%%%%%%%%%%%%%%%%%%%%%%%%%%%%%%%%%%%%%%%%%%%%%%%%%%%%%%%%%%%%%%%

\begin{proof}
(1)$\Rightarrow$(2): By Proposition \ref{prop3}, there exists an isomorphism $V\in\mathcal{B}(\mathcal{H})$ such that $T=VT^*$. Hence $T^*T=TV^*VT^*$.\\
(2)$\Rightarrow$(3) is obvious.\\
(3)$\Rightarrow$(1): By $T^*T=TN^*NT^*$ and $\mathcal{N}(T^*T)=\mathcal{N}(T)$ we get $\mathcal{N}(T^*)\subseteq\mathcal{N}(T)$. By $T^*T=TN^*NT^*$, the injectivity of $N$ and $\mathcal{N}(TN^*NT^*)=\mathcal{N}(NT^*)$ we get $\mathcal{N}(T)\subseteq\mathcal{N}(T^*)$. So $T$ is EP.\\
(1)$\Rightarrow$(4): Since (1)$\Leftrightarrow$(2) and $T$ is EP if and only if $T^*$ is EP, there exist isomorphisms $V_1$ and $V_2$ such that $T^*T=TV_1^*V_1T^*$ and $TT^*=T^*V_2^*V_2T$.\\
(4)$\Rightarrow$(1): By $T^*T=TS_1T^*$ and $\mathcal{N}(T^*T)=\mathcal{N}(T)$ we get $\mathcal{N}(T^*)\subseteq\mathcal{N}(T)$. By $TT^*=T^*S_2T$ and $\mathcal{N}(TT^*)=\mathcal{N}(T^*)$ we get $\mathcal{N}(T)\subseteq\mathcal{N}(T^*)$. Hence $T$ is EP.
\end{proof}
%%%%%%%%%%%%%%%%%%%%%%%%%%%%%%%%%%%%%%%%%%%%%%%%%%%%%%%%%%%%%%%%%%%%%%%%%%%%

\begin{remark}
\label{prop5a}
(a) If we assume that there exists $V\in\mathcal{B}(\mathcal{H})$ unitary such that $T^*T=TVT^*$, then $T$ is not in general EP. To see that let $T$ be the left shift and $V:l^2(\mathbb{N})\rightarrow l^2(\mathbb{N})$, with
$$
V\left(\left\{ x_n\right\}_{n\in\mathbb{N}}\right)=\{ x_1, x_0,
x_2, x_3, \ldots\},\text{ for all }\left\{
x_n\right\}_{n\in\mathbb{N}}\in l^2(\mathbb{N}).
$$
(b) If $T=U(A\oplus 0)U^*\in\mathcal{B}(\mathcal{H})$, with $U\in\mathcal{B}(\mathcal{K}\oplus \mathcal{L},\mathcal{H})$ unitary and $A\in\mathcal{B}(\mathcal{K})$ an isomorphism, then the solutions of the equation $T^*T=TXT^*$ are
$$X=U
\left[
\begin{array}{cc}
A^{-1}A^*A(A^*)^{-1}&B\\
C&D
\end{array}
\right]U^*,\;\;B\in\mathcal{B}(\mathcal{L},\mathcal{K}), C\in\mathcal{B}(\mathcal{K},\mathcal{L}), D\in\mathcal{B}(\mathcal{L}).$$
\end{remark}
%%%%%%%%%%%%%%%%%%%%%%%%%%%%%%%%%%%%%%%%%%%%%%%%%%%%%%%%%%%%%%%%%%%%%%%%%%%%

Using $\mathcal{N}(T^{\dag})=\mathcal{N}(T^*)$, $\mathcal{N}(T^{\dag}T)=\mathcal{N}(T)$ and $\mathcal{N}(TT^{\dag})=\mathcal{N}(T^*)$ we can show that the results of the previous proposition also hold if we replace $T^*$ with $T^{\dag}$.
%%%%%%%%%%%%%%%%%%%%%%%%%%%%%%%%%%%%%%%%%%%%%%%%%%%%%%%%%%%%%%%%%%%%%%%%%%%%

\section{Factorizations of the form $T=BC$}
%%%%%%%%%%%%%%%%%%%%%%%%%%%%%%%%%%%%%%%%%%%%%%%%%%%%%%%%%%%%%%%%%%%%%%%%%%%%

In this section we will discuss characterizations of EP operators through factorizations of the form $T=BC$, with $B$ injective with closed range and $C$ surjective.
%%%%%%%%%%%%%%%%%%%%%%%%%%%%%%%%%%%%%%%%%%%%%%%%%%%%%%%%%%%%%%%%%%%%%%%%%%%%

If $T\in\mathcal{B}(\mathcal{H})$ has closed range, then
$$
T=T\left|_{\mathcal{R}(T^*)}\right.P_{T^*},
$$
with $T\left|_{\mathcal{R}(T^*)}\right.:\mathcal{R}(T^*)\rightarrow\mathcal{H}$ injective with closed range and $P_{T^*}:\mathcal{H}\rightarrow\mathcal{R}(T^*)$ surjective. Thus if $T\in\mathcal{B}(\mathcal{H})$ has closed range, then there exist a Hilbert space $\mathcal{K}$, $B\in\mathcal{B}(\mathcal{K},\mathcal{H})$ injective with closed range and $C\in\mathcal{B}(\mathcal{H},\mathcal{K})$ surjective with $T=BC$. Note that, for any isomorphism $U\in\mathcal{B}(\mathcal{K})$, $BU\in\mathcal{B}(\mathcal{K},\mathcal{H})$ is injective with closed range, $U^{-1}C\in\mathcal{B}(\mathcal{H},\mathcal{K})$ is surjective and $(BU)(U^{-1}C)=T$ and so factorizations of that form are not unique. It is easy to see that $\mathcal{R}(T)=\mathcal{R}(B)$ and $\mathcal{N}(T)=\mathcal{N}(C)$. Moreover if $T=BC$, then obviously $T^*=C^*B^*$, with $C^*$ injective with closed range and $B^*$ surjective and $\mathcal{R}(T^*)=\mathcal{R}(C^*)$ and $\mathcal{N}(T^*)=\mathcal{N}(B^*)$.
%%%%%%%%%%%%%%%%%%%%%%%%%%%%%%%%%%%%%%%%%%%%%%%%%%%%%%%%%%%%%%%%%%%%%%%%%%%%

Assume that $T=BC$, with $B$ injective with closed range and $C$ surjective. It is easy to see that, since $B$ is injective with closed range, $B^{\dag}=(B^*B)^{-1}B^*$ and $B^{\dag}B=I_{\mathcal{K}}$ and that, since $C$ is surjective, $C^{\dag}=C^*(CC^*)^{-1}$ and $CC^{\dag}=I_{\mathcal{K}}$. Moreover
$$
T^{\dag}=C^{\dag}B^{\dag}=C^*(CC^*)^{-1}(B^*B)^{-1}B^*.
$$
%%%%%%%%%%%%%%%%%%%%%%%%%%%%%%%%%%%%%%%%%%%%%%%%%%%%%%%%%%%%%%%%%%%%%%%%%%%%

From now on we consider $T\in\mathcal{B}(\mathcal{H})$ with closed range, $B\in\mathcal{B}(\mathcal{K},\mathcal{H})$ injective with closed range and $C\in\mathcal{B}(\mathcal{H},\mathcal{K})$ surjective with $T=BC$. In the following four theorems we characterize EP operators through their factorizations of the form $T=BC$.
%%%%%%%%%%%%%%%%%%%%%%%%%%%%%%%%%%%%%%%%%%%%%%%%%%%%%%%%%%%%%%%%%%%%%%%%%%%%

\begin{theorem}
\label{theoremfact1}
The following are equivalent:
\begin{enumerate}
\item $T$ is EP.
\item $BB^{\dag}=C^{\dag}C$.
\item $\mathcal{N}(B^*)=\mathcal{N}(C)$.
\end{enumerate}
\end{theorem}
%%%%%%%%%%%%%%%%%%%%%%%%%%%%%%%%%%%%%%%%%%%%%%%%%%%%%%%%%%%%%%%%%%%%%%%%%%%%

\begin{proof}
(1)$\Leftrightarrow$(2): Since $T^{\dag}=C^{\dag}B^{\dag}$, $CC^{\dag}=I_{\mathcal{K}}$ and $B^{\dag}B=I_{\mathcal{K}}$, we have that
$$
\begin{array}{rcl}
T \text{ EP } &\Leftrightarrow&
TT^{\dag}=T^{\dag}T\\
&\Leftrightarrow& BB^{\dag}=C^{\dag}C.
\end{array}
$$
(2)$\Leftrightarrow$(3) follows from $BB^{\dag}=P_B$ and $C^{\dag}C=P_{C^*}$.
\end{proof}
%%%%%%%%%%%%%%%%%%%%%%%%%%%%%%%%%%%%%%%%%%%%%%%%%%%%%%%%%%%%%%%%%%%%%%%%%%%%

For the proof of the next theorem we will need the following lemma the proof of which is elementary and is omitted.
%%%%%%%%%%%%%%%%%%%%%%%%%%%%%%%%%%%%%%%%%%%%%%%%%%%%%%%%%%%%%%%%%%%%%%%%%%%%
\begin{lemma}
\label{lemmafact}
Let $S\in\mathcal{B}(\mathcal{H})$ and $A\in\mathcal{B}(\mathcal{H})$ with closed range. Then
$$
S=SP_A\Leftrightarrow\mathcal{N}(A^*)\subseteq\mathcal{N}(S)\text{ and }S=P_AS\Leftrightarrow\mathcal{N}(A^*)\subseteq\mathcal{N}(S^*).
$$
\end{lemma}
%%%%%%%%%%%%%%%%%%%%%%%%%%%%%%%%%%%%%%%%%%%%%%%%%%%%%%%%%%%%%%%%%%%%%%%%%%%%

\begin{theorem}
\label{theoremfact2}
The following are equivalent:
\begin{enumerate}
\item $T$ is EP.
\item $B=C^{\dag}CB$ and $C=CBB^{\dag}$.
\item $B^{\dag}=B^{\dag}C^{\dag}C$ and $C=CBB^{\dag}$.
\item $B=C^{\dag}CB$ and $C^{\dag}=BB^{\dag}C^{\dag}$.
\item $B^{\dag}=B^{\dag}C^{\dag}C$ and $C^{\dag}=BB^{\dag}C^{\dag}$.
\end{enumerate}
\end{theorem}
%%%%%%%%%%%%%%%%%%%%%%%%%%%%%%%%%%%%%%%%%%%%%%%%%%%%%%%%%%%%%%%%%%%%%%%%%%%%

\begin{proof}
By Lemma \ref{lemmafact}, $\mathcal{N}(B^{\dag})=\mathcal{N}(B^*)$ and $\mathcal{N}((C^{\dag})^*)=\mathcal{N}(C)$, we get that
$$
\begin{array}{c}
B=C^{\dag}CB\Leftrightarrow B=P_{C^*}B\Leftrightarrow\mathcal{N}(C)\subseteq\mathcal{N}(B^*),\\
C=CBB^{\dag}\Leftrightarrow C=CP_{B}\Leftrightarrow\mathcal{N}(B^*)\subseteq\mathcal{N}(C),\\
B^{\dag}=B^{\dag}C^{\dag}C\Leftrightarrow B^{\dag}=B^{\dag}P_{C^*}\Leftrightarrow\mathcal{N}(C)\subseteq\mathcal{N}(B^*),\\
C^{\dag}=BB^{\dag}C^{\dag}\Leftrightarrow C^{\dag}=P_{B}C^{\dag}\Leftrightarrow\mathcal{N}(B^*)\subseteq\mathcal{N}(C).
\end{array}
$$
Combining all the above with the equivalence of (1) and (3) in Theorem \ref{theoremfact1} we get the proof.
\end{proof}
%%%%%%%%%%%%%%%%%%%%%%%%%%%%%%%%%%%%%%%%%%%%%%%%%%%%%%%%%%%%%%%%%%%%%%%%%%%%

\begin{remark}
\label{theoremfact2a}
If we only assume that one of the conditions in (2)-(5) holds, then $T$ is not in general EP. Moreover 
$$
B=C^{\dag}CB\text{ and }B^{\dag}=B^{\dag}C^{\dag}C
$$
or
$$
C=CBB^{\dag}\text{ and }C^{\dag}=BB^{\dag}C^{\dag}
$$
do not in general imply that $T$ is EP. For the first one let $T$ be the right shift, $B=T$ and $C=I$ and for the second one let $T$ be the left shift, $B=I$ and $C=T$.
\end{remark}
%%%%%%%%%%%%%%%%%%%%%%%%%%%%%%%%%%%%%%%%%%%%%%%%%%%%%%%%%%%%%%%%%%%%%%%%%%%%

\begin{theorem}
\label{theoremfact3}
The following are equivalent:
\begin{enumerate}
\item $T$ is EP.
\item $T^*T=C^*B^*BCBB^{\dag}$ and $T^*T=C^*B^*C^{\dag}CBC$.
\item $TT^*=BCC^*B^*C^*(C^*)^{\dag}$ and $TT^*=BCBB^{\dag}C^*B^*$.
\item $T^*T=C^*B^*BCBB^{\dag}$ and $TT^*=BCC^*B^*C^*(C^*)^{\dag}$.
\item $T^*T=C^*B^*C^{\dag}CBC$ and $TT^*=BCBB^{\dag}C^*B^*$.
\end{enumerate}
\end{theorem}
%%%%%%%%%%%%%%%%%%%%%%%%%%%%%%%%%%%%%%%%%%%%%%%%%%%%%%%%%%%%%%%%%%%%%%%%%%%%

\begin{proof}
Using Theorem \ref{theoremfact2} we can show that (1) implies the rest.\\
For the converse we have that:\\
If $T^*T=C^*B^*BCBB^{\dag}$, then, using $\mathcal{N}(T^*)=\mathcal{N}(B^*)=\mathcal{N}(B^{\dag})$ and $\mathcal{N}(T^*T)=\mathcal{N}(T)$, we get that $\mathcal{N}(T^*)\subseteq\mathcal{N}(T)$
and so
$$
T^*T=C^*B^*BCBB^{\dag}\Rightarrow\mathcal{N}(T^*)\subseteq\mathcal{N}(T).
$$
Similarly we get that
$$
TT^*=BCC^*B^*C^*(C^*)^{\dag}\Rightarrow\mathcal{N}(T)\subseteq\mathcal{N}(T^*).
$$
Using $C^{\dag}C=P_{C^*}$ and $\mathcal{R}(T^*)=\mathcal{R}(C^*)$, we get that
$$
\begin{array}{rcl}
T^*T=C^*B^*C^{\dag}CBC
&\Rightarrow&
T^*T=T^*P_{T^*}T\\
&\Rightarrow&
\|Tx\|^2=\|P_{T^*}Tx\|^2,\text{ for all }x\in\mathcal{H}\\
&\Rightarrow&
Tx=P_{T^*}Tx,\text{ for all }x\in\mathcal{H}\\
&\Rightarrow&
\mathcal{R}(T)\subseteq\mathcal{R}(T^*)\\
&\Rightarrow& \mathcal{N}(T)\subseteq\mathcal{N}(T^*).
\end{array}
$$
Similarly we get that
$$
TT^*=BCBB^{\dag}C^*B^*\Rightarrow\mathcal{N}(T^*)\subseteq\mathcal{N}(T).
$$
Combining all the above we get the result.
\end{proof}
%%%%%%%%%%%%%%%%%%%%%%%%%%%%%%%%%%%%%%%%%%%%%%%%%%%%%%%%%%%%%%%%%%%%%%%%%%%%

\begin{theorem}
\label{theoremfact4}
The following are equivalent:
\begin{enumerate}
\item $T$ is EP.
\item There exists an isomorphism $V\in\mathcal{B}(\mathcal{K})$ such that $C=VB^*$.
\item There exists $N\in\mathcal{B}(\mathcal{K})$ injective with closed range such that $C=NB^*$.
\item There exist $S_1, S_2\in\mathcal{B}(\mathcal{K})$ such that $C=S_1B^*$ and $B^*=S_2C$.
\end{enumerate}
\end{theorem}
%%%%%%%%%%%%%%%%%%%%%%%%%%%%%%%%%%%%%%%%%%%%%%%%%%%%%%%%%%%%%%%%%%%%%%%%%%%%

\begin{proof}
(1)$\Rightarrow$(2): Obviously
$$
C=C|_{\mathcal{R}(C^*)}P_{C^*}
$$
and
$$
B^*=B^*|_{\mathcal{R}(B)}P_{B}.
$$
Since $T$ is EP, by Theorem \ref{theoremfact1}, $\mathcal{R}(B)=\mathcal{R}(C^*)$ and so
$$
B^*=B^*|_{\mathcal{R}(C^*)}P_{C^*}.
$$
Since $B$ is injective with closed range, $B^*$ is surjective and so
$$
B^*|_{\mathcal{R}(C^*)}:\mathcal{R}(C^*)\rightarrow\mathcal{K}
$$
is an isomorphism. Thus
$$
\left(B^*|_{\mathcal{R}(C^*)}\right)^{-1}:\mathcal{K}\rightarrow\mathcal{R}(C^*)
$$
is an isomorphism. Moreover, since $C$ is surjective,
$$
C|_{\mathcal{R}(C^*)}:\mathcal{R}(C^*)\rightarrow\mathcal{K}
$$
is also an isomorphism. From what we said so far $V:\mathcal{K}\rightarrow\mathcal{K}$ with
$$
V=C|_{\mathcal{R}(C^*)}\left(B^*|_{\mathcal{R}(C^*)}\right)^{-1}
$$
is an isomorphism and we have
$$
VB^*=C|_{\mathcal{R}(C^*)}\left(B^*|_{\mathcal{R}(C^*)}\right)^{-1}B^*|_{\mathcal{R}(C^*)}P_{C^*}=C|_{\mathcal{R}(C^*)}P_{C^*}=C.
$$
(2)$\Rightarrow$(3) and (2)$\Rightarrow$(4) are obvious.\\
(3)$\Rightarrow$(1): By the injectivity of $N$ and $C=NB^*$ we get that $\mathcal{N}(C)=\mathcal{N}(B^*)$ which, by Theorem \ref{theoremfact1}, implies that $T$ is EP.\\
(4)$\Rightarrow$(2): From $C=S_1B^*$ we get that $\mathcal{N}(B^*)\subseteq\mathcal{N}(C)$ and from $B^*=S_1C$ we get that $\mathcal{N}(C)\subseteq\mathcal{N}(B^*)$. So, by Theorem \ref{theoremfact1}, $T$ is EP.
\end{proof}
%%%%%%%%%%%%%%%%%%%%%%%%%%%%%%%%%%%%%%%%%%%%%%%%%%%%%%%%%%%%%%%%%%%%%%%%%%%%

\begin{remark}
\label{theoremfact4a}
The result is true if we replace in (2) $C=VB^*$ with $B^*=VC$ or $B^*=CV$ or $C=B^*V$. Moreover if in (3) we put $B^*=NC$ in the place of $C=NB^*$ the result is true. On the other hand if we replace $C=NB^*$ with $C=B^*N$, then $T$ is not in general EP. To see that let $T$ be the right shift, $B=T$ and $C=I$.
\end{remark}
%%%%%%%%%%%%%%%%%%%%%%%%%%%%%%%%%%%%%%%%%%%%%%%%%%%%%%%%%%%%%%%%%%%%%%%%%%%%

Using the previous theorem we can also get more general characterizations of EP operators. For example: Let $T\in\mathcal{B}(\mathcal{H})$ with closed range and assume that there exist $B\in\mathcal{B}(\mathcal{K},\mathcal{H})$ injective with closed range and $C\in\mathcal{B}(\mathcal{H},\mathcal{K})$ with $T=BC$ and an isomorphism $V\in\mathcal{B}(\mathcal{K})$ such that $C=VB^*$. Since $B$ is injective with closed range, $B^*$ is surjective and so, by $C=VB^*$, $C$ is surjective. Thus, by the previous theorem, $T$ is EP.
%%%%%%%%%%%%%%%%%%%%%%%%%%%%%%%%%%%%%%%%%%%%%%%%%%%%%%%%%%%%%%%%%%%%%%%%%%%%

\begin{Acknowledgments}
The authors wish to thank the referees for their suggestions which improved both the content and the presentation of this paper. They would also like to thank Prof. V. Rako\v{c}evi{\'c} for bringing \cite{Boasso} to their attention.
\end{Acknowledgments}
%%%%%%%%%%%%%%%%%%%%%%%%%%%%%%%%%%%%%%%%%%%%%%%%%%%%%%%%%%%%%%%%%%%%%%%%%%%%

%%%%%%%%%%%%%%%%%%%%%%%%%%%%%%%%%%%%%%%%%%%%%%%%%%%%%%%%%%%%%%%%%%%%%%%%%%%%

\end{document}